\documentclass[10pt]{amsart}
\usepackage[letterpaper, left=2cm, right=2cm, top = 2.5cm, bottom = 2cm]{geometry}
\usepackage{amsmath,amsfonts,amsthm,bbm, amssymb}
\usepackage[arrow,matrix]{xy}
\usepackage{enumerate}

\usepackage[usenames]{xcolor}
\usepackage{comment}
\usepackage{graphicx}
\usepackage{amssymb}

\theoremstyle{plain}
\numberwithin{equation}{section}
\newtheorem{thm}{Theorem}[section]
\newtheorem{cor}[thm]{Corollary}
\newtheorem{lem}[thm]{Lemma}
\newtheorem{prop}[thm]{Proposition}

\newtheorem*{prop*}{\scshape Proposition}
\newtheorem*{thm*}{\scshape Theorem}

\newtheorem*{lem*}{\scshape Lemma}
\newtheorem*{cor*}{\scshape Corollary}

\theoremstyle{definition}
\newtheorem{df}[thm]{Definition}
\newtheorem{ex}[thm]{Example}
\newtheorem{rmk}[thm]{Remark}

	\newcommand{\Z}{\mathbb{Z}}
	
    \newcommand{\Cc}{\mathcal{C}}

	\renewcommand{\P}{\mathbb{P}}

	\newcommand{\Sm}{\mathbf{Sm}}

	\newcommand{\Sch}{\mathbf{Sch}}
    \newcommand{\PST}{\mathbf{PST}}
    \newcommand{\REC}{\mathbf{REC}}
    \newcommand{\Ab}{\mathbf{Ab}}
	
\providecommand{\frac}[1]{\operatorname{Frac}(#1)}  
\providecommand{\Spec}[1]{\operatorname{Spec}(#1)} 	%
\newcommand{\Sing}{\operatorname{Sing}}              
\newcommand{\chark}{\operatorname{char}}		
  
\newcommand{\Pic}{\operatorname{Pic}}		 
\newcommand{\dv}{\operatorname{div}} 				 
\newcommand{\tr}{\operatorname{Tr}}					 
\renewcommand{\ker}{\operatorname{Ker}}				 
\renewcommand{\H}{\operatorname{H}}				     
\newcommand{\CH}{\operatorname{CH}}					 
\renewcommand{\lim}{\operatorname{lim}}			 

	\def\Cb{\ol{C}}
	
	\renewcommand{\cong}{\simeq}
	
	\newcommand{\ol}{\overline}
    \newcommand{\bs}{\backslash}
    \newcommand{\lto}{\longrightarrow}

\def\Xb{\ol{X}}
\def\Xb{\overline{X}}

\def\etab{\overline{\eta}}

\newcommand{\cO}{\mathcal{O}}

\usepackage[pdfauthor={Federico Binda, Jin Cao, Wataru Kai and Rin Sugiyama}, %
				pdftitle={Torsion and divisibility for reciprocity sheaves and 0-cycles with modulus}%
			dvips,colorlinks=true]{hyperref}
\hypersetup{
	bookmarksnumbered=true,
	linkcolor=black,
	citecolor=blue,
}
\newcounter{elno}   
\newenvironment{romanlist}{
                         \begin{list}{\roman{elno})
                                     }{\usecounter{elno}}
                      }{
                         \end{list}}

\begin{document}
\author{Federico Binda, Jin Cao, Wataru Kai and Rin Sugiyama}
\title{Torsion and divisibility for reciprocity sheaves and 0-cycles with modulus}
\subjclass[2010]{Primary 14C25; Secondary 19E15, 14F42}

\AtEndDocument{\bigskip{\footnotesize%
  (F.~Binda) \textsc{Fakult\"at f\"ur Mathematik,  Universit\"at Duisburg-Essen, Thea-Leymann Strasse 9, 45127 Essen, Germany.} 
  \textit{E-mail address},  \texttt{federico.binda@uni-due.de} \par
  \addvspace{\medskipamount}
  (J.~Cao) \textsc{Fakult\"at f\"ur Mathematik,  Universit\"at Duisburg-Essen, Thea-Leymann Strasse 9, 45127 Essen, Germany.}
  \textit{E-mail address},   \texttt{jin.cao@stud.uni-due.de} \par
 \addvspace{\medskipamount}
  (W.~Kai) \textsc{Graduate School of Mathematical Sciences, the University of Tokyo, 3-8-1 Komaba, Meguro-ku, 153-8914 Tokyo, Japan.} 
  \textit{E-mail address},   \texttt{kaiw@ms.u-tokyo.ac.jp} \par
  \addvspace{\medskipamount}
  (R.~Sugiyama) \textsc{Department of Mathematics, Tokyo Denki University, 8 Senju-Asahi-Cho, Adachi-Ku, 120-8551 Tokyo, Japan.}
  \textit{E-mail address},   \texttt{rin@mail.dendai.ac.jp} \par
  }}

\begin{abstract}
The notion of {\it modulus} is a striking feature of Rosenlicht-Serre's theory of generalized Jacobian varieties of curves. It was carried over to algebraic cycles on general varieties by Bloch-Esnault, Park, R\"ulling, Krishna-Levine. Recently, Kerz-Saito introduced a notion of Chow group of $0$-cycles with modulus in connection with geometric class field theory with wild ramification for varieties over finite fields. 
We study the non-homotopy invariant part of the Chow group of $0$-cycles with modulus and show their torsion and divisibility properties. 

Modulus is being brought to sheaf theory by Kahn-Saito-Yamazaki in their attempt to construct a generalization of Voevodsky-Suslin-Friedlander's theory of homotopy invariant presheaves with transfers.
We prove parallel results about torsion and divisibility properties for them.
\end{abstract}

\maketitle

\setcounter{tocdepth}{1}

\section{Introduction}

Let $k$ be a field and let $\Xb$ be a proper $k$-variety equipped with an effective Cartier divisor $D$.
For such a  pair $(\Xb, D)$,  Kerz and Saito recently defined in \cite{KS1} a notion of Chow group $\CH_0(\Xb|D)$ of $0$-cycles on $\Xb$ with modulus $D$  as a quotient of the group $Z_0(X)$ of $0$-cycles on the open complement $X:=\Xb\setminus |D|$. 
When $\Xb$ is a smooth projective curve, the group $\CH_0(\Xb|D)$ is isomorphic to the relative Picard group $\Pic (\Xb ,D)$ of isomorphism classes of pairs given by a line bundle on $\Xb$ together a trivialization along $D$. Its degree-0-part agrees with the group of $k$-rational points of the  generalized Jacobian $\mathrm{Jac} (\Xb|D)$ of Rosenlicht and Serre (see, for instance, \cite[Chapter II]{SerreGACC}). If $D$ is non-reduced, then $\mathrm{Jac}(\Xb |D)$ is a commutative algebraic group of general type, i.e.~an extension of a semi-abelian variety by a unipotent group, which depends on the multiplicity of $D$. The existence of such non-homotopy invariant part suggests that the group $\CH_0(\Xb|D)$ may give new geometric and arithmetic information about the pair $(\Xb, D)$ that cannot be captured by the classical (homotopy invariant) motivic cohomology groups.

Intimately connected with the world of cycles subject to some modulus  conditions is the recent work of Kahn, Saito and Yamazaki \cite{KSY}, which gives a  categorical attempt at the quest for a non-homotopy-invariant motivic theory. This encompasses unipotent phenomena and is modeled on the generalized Jacobians of Rosenlicht and Serre. Kahn-Saito-Yamazaki developed the notion of ``reciprocity'' for (pre)sheaves with transfers, which is weaker than homotopy invariance, with the purpose of eventually constructing a new motivic triangulated category, larger than Voevodsky's $\mathbf{DM}^{\rm eff}(k,\Z)$ and containing unipotent information.

The goal of this paper is to exhibit some differences between the classical
 homotopy invariant objects and the new non-homotopy invariant ones, such as $0$-cycles with modulus and reciprocity sheaves. 

For $0$-cycles, we shall see in \S \ref{rel-to-Suslin} that there is a canonical surjection from  the Chow group with modulus to the $0$-th Suslin homology group (as defined e.g.~ in \cite[Definition 7.1]{MVW}) %
\[
\pi_{\Xb,D}\colon \CH_0(\Xb|D)\lto \H_0^{\Sing}(X).
\]
Since $\H^{\Sing}_0(X)$  is the maximal homotopy invariant quotient of the group $Z_0(X)$ of $0$-cycles on $X$, the kernel $\operatorname{U}(\Xb|D)$ of $\pi_{\Xb,D}$ measures the failure of $\CH _0(\ol{X}|D)$ to be homotopy invariant (nonetheless, its degree-$0$-part enjoys  $\mathbb{P}^1$-invariance as pointed out in Remark \ref{P1inv}).
The first result of this paper is the following divisibility property of $\operatorname{U}(\Xb|D)$:

\begin{thm}[{\it see} Theorem \ref{div-chow}]\label{div-chow-intro}
\begin{enumerate}[{\rm (1)}]
\item
If $\chark (k)=0$, then the group $\operatorname{U}(\ol{X}|D)$ is divisible.
\item
If $\chark (k)=p>0$, then $\operatorname{U}(\Xb|D)$ is a $p$-primary torsion group.
\end{enumerate}
\end{thm}


Our second result is complementary to Theorem \ref{div-chow-intro}, and concerns the torsion part of $\CH_0(\Xb|D)$.
\begin{thm}[{\it see} Corollary \ref{Torsion-Curves-Body}]\label{Torsion-curves-Intro}
Let $k$ be an algebraically closed field of exponential characteristic $p\geq 1$.
Let $\Xb$ be a projective variety over $k$, regular  in codimension one. Let $D$ be an effective Cartier divisor on $\Xb$ such that the open complement $X= \Xb \setminus |D|$ is smooth over $k$. Let $\alpha\in \CH_0(\Xb|D)$ be a prime-to-$p$-torsion cycle. Then there exist a smooth projective curve $\Cb$ and a morphism $\varphi \colon \Cb \to \Xb$ for which $\varphi ^{*}D$ is a well defined Cartier divisor on $\Cb$, and a prime-to-$p$-torsion cycle $\beta\in \CH_0(\Cb| (\varphi ^*D)_{\mathrm{red}})$ such that $\varphi _*(\beta)=\alpha$.
\end{thm}

In other words, the prime-to-$p$-torsion part of $\CH_0(\Xb|D)$ is nearly independent of the multiplicities of $D$. However, the theorem  does not provide, a priori, an identification between the prime-to-$p$-torsion parts of $\CH_0(\Xb|D)$ and of $\H_0^{\Sing}(X)$.

Our third result is about reciprocity sheaves:

\begin{thm}[{\it see} Theorem \ref{htpy}]\label{Theorem-htpy-Intro}
Let $F \in \mathbf{REC}_k$ be a reciprocity presheaf with transfers, separated for the Zariski topology.
Then $F$ is homotopy invariant (i.e.~the map of presheaves $F\to F(-\times \mathbb{A}^1)$ is an isomorphism) either when 
$\chark (k)=0$ and $F$ is torsion, or when
$\chark (k)=p>0$ and $F$ is $p$-torsion free.
\end{thm}

In order to measure the lack of homotopy invariance of a reciprocity sheaf $F$, we define, similarly to what we did for $0$-cycles, $\operatorname{U}(F)$ to be the kernel of the canonical surjection 
\[
F\to \H_0(F),
\]
where $\H_0(F)$ is the maximal homotopy invariant quotient of $F$ (see Section \ref{unip-section}). 
Corollary  \ref{cor1-intro} gives an analogue of  Theorem \ref{div-chow-intro} for  the reciprocity sheaf $\operatorname{U}(F)$:

\begin{cor}[{\it see} Corollary \ref{cor1}]\label{cor1-intro}
\begin{enumerate}[{\rm (1)}]
\item
If $\chark (k)=0$, then $\operatorname{U}(F)$ is divisible.

\item
If $\chark (k)=p>0$, then $\operatorname{U}(F)$ is a $p$-primary torsion sheaf.
\end{enumerate}
\end{cor}
 
We remark that by combining Corollary \ref{cor1-intro} and some results of \cite{KSY}, one can give an alternative proof of Theorem \ref{div-chow-intro} when $X$ is smooth and quasi-affine.

\medskip

This paper is organized as follows.
Section \ref{Section-zero-cycles} is devoted to studying the Chow groups of $0$-cycles with modulus.
In \S\,\ref{rel-to-Suslin}, we investigate the relation between $\CH_0(\Xb|D)$ for a pair $(\Xb, D)$ and the $0$-th Suslin homology of the complement $\Xb \setminus |D|$. 
In \S\,2.3, we prove Theorem \ref{div-chow-intro},  using some technical results, Lemma \ref{key-lem} and Lemma \ref{key-lem-p}.
In \S \S 2.4--2.5, we prove Theorem \ref{Torsion-curves-Intro}. 
Its proof is purely geometric and follows the approach of Levine \cite{MarcTorsion} to Rojtman's torsion theorem for singular projective varieties. One of the main tools, inspired by the work \cite{LW} of Levine and Weibel on $0$-cycles on singular varieties, is a rigidity result for the torsion subgroup of $\CH_0(\Xb|D)$ (see Theorem \ref{thm-discreteness}).

Section \ref{Section-Reciprocity} is devoted to studying torsion and divisibility phenomena for reciprocity (pre)sheaves with transfers.
In \S\,\ref{section-htpy-rec}, we prove Theorem \ref{Theorem-htpy-Intro}, using again Lemma \ref{key-lem} and Lemma \ref{key-lem-p}.
In \S\,\ref{unip-section}, we study the sheaf $\operatorname{U}(F)$ and the homology sheaves $\H _i(F)$ of the Suslin complex of  a reciprocity sheaf $F$.
As consequences of Theorem \ref{Theorem-htpy-Intro} we get Corollary \ref{cor1-intro} and some further result on $\H_i(F)$ (see Corollary \ref{cor3}).

\subsection*{Acknowledgments}Large part of this work was carried out during the special semester in Motivic Homotopy Theory at the University of Duisburg-Essen (SS 2014). The  authors wish to thank  Marc Levine heartily for providing an excellent working environment  and for the support via the Alexander von Humboldt foundation and the SFB Transregio 45 ``Periods, moduli spaces and arithmetic of algebraic varieties''. The first author is supported by the DFG Schwerpunkt Programme 1786 "Homotopy theory and Algebraic Geometry". 
The third author is supported by JSPS as a Research Fellow and through JSPS KAKENHI Grant (15J02264), and was supported by the Program for Leading Graduate Schools, MEXT, Japan during the work.
The fourth author is supported by JSPS KAKENHI Grant (16K17579).
We sincerely appreciate the referee's careful and valuable comments to an earlier draft of this paper,  which helped us to significantly clarify and improve the exposition.

\section{Chow group of $0$-cycles with modulus}\label{Section-zero-cycles}
\subsection{Definition of $0$-cycles with modulus}

We recall the definition of $\CH _0(\ol{X}|D)$ from  Kerz and Saito \cite{KS1}. 

\subsubsection{}\label{DefChowMod}We fix a base field $k$. For an integral scheme $\ol{C}$ over $k$ and for a closed subscheme $E$ of $\ol{C}$, we set
\begin{align*}
G(\ol{C},E)
=\bigcap_{x\in E}\mathrm{Ker}\bigl(\cO_{\ol{C},x}^{\times}\to \cO_{E,x}^{\times} \bigr) = \varinjlim_{E\subset U\subset \ol{C}}\Gamma(U, \ker(\cO_{\ol{C}}^\times \to \cO_{E}^\times)),
\end{align*}
where $x$ runs over all the points of $E$ and $U$ runs over the set of open subsets containing $E$. The intersection takes place in the rational function field $k(\overline{C})$. We say that a rational function $f\in G(\ol{C},E)$ satisfies the modulus condition with respect to $E$.

\subsubsection{}\label{DefChowMod1}Let $\ol{X}$ be a scheme of finite type over $k$ and let $D$ be an effective Cartier divisor on $\ol{X}$. Write $X$ for the complement $\ol{X} \setminus |D|$ and $Z_0(X)$ for the group of $0$-cycles on $X$. Let $\ol{C}$ be an integral normal curve over $k$ and
let $\varphi_{\ol{C}}:\ol{C}\to \ol{X}$ be a finite morphism such that $\varphi_{\ol{C}}(\ol{C})\not \subset D$. Write  $C= \varphi_{\ol{C}}^{-1}(X)$. The push-forward of cycles gives a group homomorphism
$ \tau_{\ol{C}}\colon G(\ol{C},\varphi_{\ol{C}}^*(D))\to Z_0(X)$ that sends a function $f$ to the $0$-cycle $(\varphi _{\overline{C}}|_{C})_*\dv_{\ol{C}}(f)$.

\begin{df}\label{DefChowMod-Definition}
In the notations of \ref{DefChowMod1}, we define the \emph{Chow group $\CH_0(\ol{X}|D)$ of $0$-cycles of $\ol{X}$ with modulus $D$} as the cokernel of the homomorphism 
\[
\tau\colon\bigoplus_{\varphi_{\ol{C}}\colon \ol{C}\to \ol{X}}G(\ol{C},\varphi_{\ol{C}}^*(D)) \xrightarrow{\bigoplus \tau _{\ol{C}}} Z_0(X),
\]
where the sum runs over the set of finite morphisms from an integral normal curve $\varphi_{\ol{C}}\colon \ol{C}\to \ol{X}$ such that $\varphi_{\ol{C}}(\ol{C})\not \subset D$.
\end{df}
\begin{rmk} A generalization to higher dimensional cycles and to  higher Chow groups (in the sense of Bloch) $\CH_r(\ol{X}|D,n)$ is given in \cite{BS}, where the above groups $\CH _0(\ol{X}|D)$ are shown to agree with the corresponding higher Chow groups with modulus $\CH_0(\ol{X}|D, 0)$ (see \cite[Theorem 3.3]{BS}). 
A different definition of Chow group of $0$-cycles with modulus is proposed by  Russell in  \cite{R}.
\end{rmk}


\begin{prop} \label{example}
Let $(\ol{X},D)$ and $(\ol{Y},E)$ be pairs of proper schemes of finite type over $k$ and effective Cartier divisors on them.
Assume that $\ol{Y}$ is connected and $Y= \ol{Y}\setminus |E|$ has a $k$-rational point.
If the degree map induces an isomorphism $\CH _0(\ol{Y}_{k'}|E_{k'})\xrightarrow{\simeq } \Z $ for any finite field extension $k'/k$, then the proper push-forward map induced by the projection $p_1\colon \ol{X}\times \ol{Y}\to \ol{X}$ is an isomorphism
\[
p_{1,*}\colon \CH_0(\ol{X}\times \ol{Y}|\ol{X}\times E+D\times \ol{Y})\xrightarrow{\simeq } \CH_0(\ol{X}|D).
\]
\end{prop}

\begin{proof}
Let $y_0$ be a $k$-rational point of $Y$ and let $\iota \colon \ol{X}\times \{ y_0 \} \hookrightarrow \ol{X}\times \ol{Y}$ be the canonical closed embedding.
Since we have that $p_{1,*}\circ \iota _* =\mathrm{id}$ on $\mathrm{CH}_0(\ol{X}|D)$, it suffices to show that $\iota _*\colon \CH _0(\ol{X}|D)\to \CH _0(\ol{X}\times \ol{Y}|\ol{X}\times E+D\times \ol{Y})$ is surjective.

Let $z$ be a closed point on $X\times Y$ (here we write $X=\ol{X}\setminus |D|$, $Y=\ol{Y}\setminus |E|$) and let $k(z)$ be its residue field. We claim that the class of $z$ in $\CH _0(\ol{X}\times \ol{Y}|\ol{X}\times E+D\times \ol{Y})$ comes from $\CH _0(\ol{X}\times \{ y_0 \} |D\times \{ y_0\} )$.
Note that the $0$-cycle $z$ comes from a canonical $0$-cycle on $(\ol{X}\times \ol{Y})_{k(z)}$. By the commutative diagram of push-forward maps below, it suffices to show that this $0$-cycle comes from $\CH _0((\ol{X}\times \{ y_0\} )_{k(z)} |(D\times \{ y_0\} )_{k(z)} )$,
\[ \xymatrix{
\CH _0((\ol{X}\times \{ y_0\} )_{k(z)} |(D\times \{ y_0\} )_{k(z)} ) \ar[r] \ar[d]
& \CH _0((\ol{X}\times \ol{Y})_{k(z)} |(\ol{X}\times E+D\times \ol{Y})_{k(z)} ) \ar[d]
\\
\CH _0(\ol{X}\times \{ y_0\} |D\times \{ y_0\} )\ar[r] 
&\CH _0(\ol{X}\times \ol{Y} |\ol{X}\times E+D\times \ol{Y} ) .
}   \]
Thus by replacing $k$ by $k(z)$, we may assume $z$ is a rational point $x\times y$, where $x\in X(k)$ and $y\in Y(k)$ (note that the assumptions on $\ol{Y}$ remain valid after this replacement).
Since we have $\CH _0(\ol{Y}|E)\cong \Z $ via the degree map, there are finitely many integral normal curves $\ol{W}_i$ with finite maps $\varphi _i\colon \ol{W}_i\to \ol{Y}$ and rational functions $f_i$ in $G(\ol{W}_i,\varphi _i^{-1}(E))$ such that the equality of cycles
\[ \sum _i \varphi _{i,*}\dv _{\ol{W}_i}(f_i) =[y]-[y_0]   \]
holds on $\ol{Y}$.
Let $\ol{T}_i=\{ x \} \times W_i ~(\simeq W_i)$ and let $\psi _i=(x,\varphi _i)\colon \ol{T}_i\to \ol{X}\times \ol{Y}$ be the induced finite map. Then we find that $f_i$ belongs to $G(\ol{T}_i,\psi _i^{-1}(\ol{X}\times E+D\times \ol{Y}))$ and the following equality holds on $X\times Y$
\[ \sum _i \psi _{i,*}\dv _{\ol{T}_i}(f_i)=[x\times y]-[x\times y_0] .  \]
Thus the class of $z$ is in the image of $\iota _*$. This completes the proof.
\end{proof}

\begin{rmk}\label{P1inv}
A relevant example for Proposition \ref{example} is the isomorphism 
\begin{equation}\label{eq:P1invariance} \CH_0(\ol{X}\times \mathbb{P}^1|D\times \mathbb{P}^1)\simeq \CH_0(\ol{X}|D), \end{equation}
that can be interpreted as a $\mathbb{P}^1$-invariance property for Chow groups of $0$-cycles with modulus. For $\Xb$ smooth and quasi-projective, the isomorphism \eqref{eq:P1invariance} is also a consequence of \cite[Theorem 4.6]{KP3}.
\end{rmk}

\begin{rmk}
Proposition \ref{example} can be interpreted in the language of reciprocity sheaves (see Section 3.1) as follows:
Let $(\ol{X},D)$ and $(\ol{Y},E)$ be pairs of proper integral $k$-schemes and effective Cartier divisors such that $\ol{X}\setminus |D|$ and $\ol{Y}\setminus |E|$ are smooth and quasi-affine.
For such pairs, we have reciprocity presheaves $ h(\ol{X},D)$ and $ h(\ol{Y},E)$ (see Remark \ref{Phi-def}) which, for any field extension $k'/k$, satisfy
$ h(\ol{X},D)(k')=\CH _0(\ol{X}_{k'},D_{k'})$ and $ h(\ol{Y},E)(k')=\CH _0(\ol{Y}_{k'},E_{k'})$ (see \ref{Phi-def} as well as \cite[Proposition 2.2.2]{KSY}).
Now assume that $Y:=\ol{Y}\setminus |E|$ has a $k$-rational point and that $ h(\ol{Y},E)_{\mathrm{Zar}}\simeq \Z$. In particular, for any field extension $k'/k$ we have $\CH _0(\ol{Y}_{k'}|E_{k'})\cong \mathbb{Z}$.
	An example of such pair is given by  $(\P^1,\infty)$.
	Then there is an isomorphism
	\begin{equation}\label{interpreted}
	h(\ol{X} \times \ol{Y}, D \times \ol{Y}+\ol{X}\times E)_{\mathrm{Zar}} \xrightarrow{\simeq } h(\ol{X},D )_{\mathrm{Zar}}.
	\end{equation}
Indeed, by Proposition \ref{example} we have isomorphisms $ h(\ol{X} \times \ol{Y}, D \times \ol{Y}+\ol{X}\times E)(k') \xrightarrow{\simeq }  h(\ol{X},D )(k')$ for any field extension $k'/k$. Then, the Injectivity Theorem \cite[Theorem 6]{KSY} for reciprocity sheaves applied to the kernel and cokernel of the map \eqref{interpreted} gives our assertion.

	Note that the condition $ h(\ol{Y},E)_{\mathrm{Zar}}\simeq \Z$ implies that $ h(\ol{Y},E')_{\mathrm{Zar}}\simeq  h_0(Y)_{\mathrm{Zar}}\simeq \Z$ for every divisor $E'$ contained in $E$ as a subscheme.
The reader should compare the isomorphism with the isomorphism of homotopy invariant sheaves 
\[
 {h}(X\times Y)_{\mathrm{Zar}} \simeq  {h}(X)_{\mathrm{Zar}}.
\]
	The displayed isomorphism \eqref{interpreted} will give some examples to the question raised in  \cite[Remark 3.5.1]{KSY}, e.g.,~if $\dim X=1$, we get an isomorphism
	\[
		h_{0}(\ol{X} \times \ol{Y}, (D \times \ol{Y}+\ol{X}\times E)_\mathrm{red})_{\mathrm{Zar}} \cong h_{0}(X\times Y)_{\mathrm{Zar}}.
	\]
\end{rmk}

\subsection{Relation to Suslin homology}\label{rel-to-Suslin}
Let $S$ be an irreducible scheme of finite type over $k$ and $X$ be a scheme of finite type over $S$. We denote by $C_0(X/S)$ the group of finite correspondences of $X$ over $S$ \cite[\S 3]{SV}, i.e.~the free abelian group generated by closed integral subschemes of $X$ that are finite and surjective over $S$. Recall from \cite[\S 3]{SV} (or \cite[Definition 7.1]{MVW}) that one defines the $0$-th Suslin homology group $\H_0^{\Sing}(X)$ of $X$ to be the cokernel of 
\[C_0(X\times (\mathbb{P}^1\setminus\{1\} )/ (\mathbb{P}^1\setminus\{1\} ))\xrightarrow{\partial = (\partial_0 - \partial_\infty)} C_0(X/\Spec{k}) = Z_0(X), \]
where $\partial_i$ is induced by $t=i\colon \Spec{k}\to \mathbb{P}^1$, for $i=0, \infty$. 
The groups $\H _0^{\Sing}(X)$ are covariant for arbitrary morphisms of $k$-schemes of finite type.
Note that there is a natural surjection induced by the identity map on $0$-cycles:
$ \H_0^{\Sing}(X)\to \CH_0(X).$
The following is stated in \cite[Introduction]{KS1}.
We include a verification of it for the convenience of the reader.

\begin{prop}\label{ch-to-suslin}
Let $\ol{X}$ be a proper scheme over $k$, $D$ an effective Cartier divisor on $\ol{X}$ and $X$ the complement $\ol{X}\setminus |D|$. 
Then the identity map of $Z_0(X)$ induces a natural surjection
\[
\pi_{\ol{X},D}\colon\CH_0(\ol{X}|D) \lto \H_0^{\Sing}(X).
\]
\end{prop}
\begin{proof}The two groups have the same set of generators, so it is enough to show that the relations defining the Chow group of $0$-cycles with modulus are $0$ in the Suslin homology group.
Let $\varphi_{\ol{C}}:\ol{C}\to \ol{X}$ be a finite morphism from a normal curve $\ol{C}$ with $\varphi_{\ol{C}}(\ol{C})\not \subset D$ and let $f\in G(\ol{C},\varphi_{\ol{C}}^*(D))$. We claim that
$\tau_{\ol{C}}(f)=0$ 
in $\H_0^{\Sing}(X)$.
We may replace $\ol{X}$ by $\ol{C}$ to prove the claim (by the covariance of $\H _0^{\Sing}(-)$).
We regard $f$ as a morphism $f\colon\ol{C}\to \P^1$. 
Since $\ol{C}$ is proper over $k$ the map $f$ either constant or surjective. In the former case the claim is obvious, so let us assume $f$ is surjective.
Let $\Gamma_f \subset \ol{C} \times \mathbb{P}^1$ be the graph of $f$ and let $W=\Gamma_f\cap (\ol{C}\times (\P^1\bs\{1\}))$.
Since $f\equiv 1 \mod \ \varphi _{\ol{C}}^*(D)$, the irreducible closed set $W$ belongs to $C_0((\ol{C}\setminus {D})\times (\mathbb{P}^1\setminus\{1\} )/ (\mathbb{P}^1\setminus\{1\} ))$ and that we have
\[
\partial(W) =(\partial_0-\partial_\infty) (W)=\dv_{\ol{C}}(f)=\tau_{\ol{C}}(f),
\]
proving the claim.
\end{proof}

\subsection{Divisibility result for $0$-cycles with modulus}

Let $\Xb$ be a proper $k$-scheme and let $D$ be an effective Cartier divisor on it. As before, let $X = \Xb \setminus |D|$. 
By Proposition \ref{ch-to-suslin}, there is a canonical surjection 
\[
\pi_{\ol{X},D}\colon\CH_0(\ol{X}|D) \lto \H_0^{\Sing}(X).
\]
We define $\operatorname{U}(\ol{X}|D)$ to be the kernel of $\pi_{\ol{X},D}$, and call it \emph{the unipotent part of $\CH_0(\ol{X}|D)$}.
Since the surjection $\pi_{\ol{X},D}$ is compatible with the degree maps, 
the group $\operatorname{U}(\ol{X}|D)$ fits into the following exact sequence
\[
0\lto \operatorname{U}(\ol{X}|D) \lto \CH_0(\ol{X}|D)^0\lto  \H_0^{\Sing}(X)^0\lto 0,
\] 
where $\CH_0(\Xb|D)^0$ and $\H_0^{\Sing}(X)^0$ denote the degree-0-parts of $\CH_0(\Xb|D)$ and $\H_0^{\Sing}(X)$ respectively.
Since $ \H_0^{\Sing}(X)$ is the maximal homotopy invariant quotient of the group of $0$-cycles $Z_0(X)$, the group $\operatorname{U}(\ol{X}|D)$ measures precisely the failure of homotopy invariance of $\CH _0(\Xb |D)$.

\begin{thm}\label{div-chow}
Let $\Xb$ be a proper $k$-scheme and $D$ be an effective Cartier divisor on it. Then we have:
\begin{enumerate}[{\rm (1)}]
\item
If $\chark (k)=0$, then $\operatorname{U}(\ol{X}|D)$ is divisible.
\item
If $\chark (k)=p>0$, then $\operatorname{U}(\ol{X}|D)$ is a $p$-primary torsion group.
\end{enumerate}
\end{thm}

We start with some auxiliary lemmas.
\begin{lem}\label{key-lem}
Let $k$ be a field of characteristic zero.
Let $\ol{C}$ be a proper normal integral curve over $k$.
Let $D$ be an effective Cartier divisor on $\ol{C}$ and write $D_\mathrm{red}$ for the corresponding  reduced divisor. 
Then the quotient group $G(\ol{C},D_\mathrm{red})/G(\ol{C},D)$ has a $k$-vector space structure.
In particular, for any integer $n>0$, there is an isomorphism
$
G(\ol{C},D)/n \stackrel{\simeq}{\longrightarrow} G(\ol{C},D_\mathrm{red})/n.
$
\end{lem}
\begin{proof}
Write $D=\sum_{i=1}^rn_i[P_i]$. 
By the definition of $G(\ol{C},D)$, one has the following commutative diagram with exact rows {\small
\[
\xymatrix{
0\ar[r]& G(\ol{C},D)\ar[r]\ar[d]&\ar@{=}[d] \cO_{\ol{C},D_\mathrm{red}}^{\times }\ar[r]& \ar[d]\displaystyle \bigoplus_{i=1}^r \bigl(\cO_{\ol{C},P_i}/\mathfrak{m}_{P_i}^{n_i}\bigr)^{\times }\ \ar[r]&0\\
0\ar[r] &G(\ol{C},D_\mathrm{red})\ar[r]& \cO_{\ol{C},D_\mathrm{red}}^{\times }\ar[r]& \displaystyle \bigoplus_{i=1}^r k(P_i)^{\times} \ar[r]&0.
}
\]}
Therefore by the snake lemma, we get
\begin{equation}\label{eq:lem-key-lem}
G(\ol{C},D_\mathrm{red})/G(\ol{C},D)\xleftarrow{\simeq} \bigoplus_{i=1}^r \frac{1+\mathfrak{m}_{P_i}}{1+\mathfrak{m}_{P_i}^{n_i}}\xrightarrow{\simeq}\bigoplus_{i=1}^r \mathfrak{m}_{P_i}/\mathfrak{m}_{P_i}^{n_i},
\end{equation}
where the second isomorphism is obtained by taking the logarithm. 
Since the last term in \eqref{eq:lem-key-lem} is a $k$-vector space, the group $G(\ol{C},D_\mathrm{red})/G(\ol{C},D)$ has an induced $k$-vector space structure.
\end{proof}

\begin{lem}\label{key-lem-p}
Let $k$ be a field of positive characteristic $p$.
Let $\overline{C}$ be an integral scheme of finite type over $k$
and $D'\subset D$ be closed subschemes of $\overline{C}$ having the same support.
Then there is a positive integer $m$ such that for any
 $f\in G(\ol{C},D')$,
its $p^m$-power $f^{p^m}$ belongs to $G(\ol{C}, D)$ (and consequently, the quotient group $G(\ol{C},D')/G(\ol{C},D)$ is annihilated by a power of $p$).
\end{lem}

\begin{proof}
Let $f\in G(\overline{C},D')
= \bigcap _{x\in D'} \ker (\mathcal{O}^*_{\ol{C},x}\to \mathcal{O}^*_{D',x})$.
For each point $x\in D'$, we have 
\[ f \in 1+ I'_x \subset \mathcal{O}_{\overline{C},x}^* \]
where $I'_x$ is the stalk at $x$ of the ideal sheaf 
$I'\subset \mathcal{O}_{\overline{C}}$ defining $D'$.
By the relation $|D|= |D'|$, 
the defining ideal $I$ of 
$D$ contains some power of $I'$ (say $(I')^{p^m}\subset I$). 
Thus we have
\[ f^{p^m} \in 1+ I'^{p^m}_x \subset 1+I_x \]
for each $x\in |D|\subset |D'|$.
Therefore $f^{p^m}$ belongs to 
$\bigcap _{x\in D} (1+I_x )= G(\overline{C},D)$.
This completes the proof.
\end{proof}


\begin{lem}\label{gen-U}
There is a surjection
\begin{equation}\label{thereIsASurj}
\bigoplus \varphi _*: \bigoplus_{\varphi\colon\ol{C}\to\ol{X}}G(\ol{C},\varphi^*(D)_\mathrm{red})/G(\ol{C},\varphi^*(D))\lto \operatorname{U}(\ol{X}|D)
\end{equation}
where $\varphi : \ol{C}\to\ol{X}$ runs over the set of finite morphisms from normal proper curves $\ol{C}$ over $k$ such that $\varphi_{\ol{C}}(\ol{C}) \not \subset D $.
\end{lem}
\begin{proof}
By definition, the group $\operatorname{U}(\ol{X}|D)$ is generated by the cycles of the form $\partial (W)$ for $W\in C_0 (X\times (\P^1\setminus\{ 1\} )/(\P^1\setminus\{ 1\} ))$.
Without loss of generality, we may assume that $W$ is irreducible. Let $\ol{W}$ be its closure in $\ol{X}\times \mathbb{P}^1$ and $\ol{W}^N$ be its normalization. Note that it is an integral normal curve. 
Let $f$ be the composite map $f\colon \ol{W}^N\to \ol{X}\times \P ^1\to  \P^1$ and let $\varphi\colon\ol{W}^N\to \ol{X}\times \P ^1 \to \ol{X}$. From the condition $W\in C_0(X\times (\P^1\setminus\{ 1\} )/(\P^1\setminus\{ 1\} ))$, we find $f\in G(\overline{W}^N,\varphi ^*(D)_{\mathrm{red}})$.
We have 
\begin{equation}\label{partialW}
\partial(W)=\varphi _*(\dv_{\ol{W}^N}(f)).
\end{equation}
Since $\ol{W}^N$ is a proper integral curve, the map $\varphi \colon \ol{W}^N\to \ol{X}$ is either constant or finite. In the former case the right hand side of the equation \eqref{partialW} is zero.
In the latter case, the finite map $\varphi $ and the function $f\in G(\ol{W}^N,\varphi ^*(D)_{\mathrm{red}})$ determine an element in the source of the map \eqref{thereIsASurj}.
In any case the equation \eqref{partialW} displays $\partial (W)$ as an element in the image of the map \eqref{thereIsASurj}.
\end{proof}

\begin{proof}[{\bf Proof of Theorem \ref{div-chow}}]
Given the previous Lemma \ref{gen-U}, if $\chark (k) =0$ then the statement is  a consequence of Lemma \ref{key-lem}. 
If $\chark (k) = p>0$, it is a consequence of Lemma \ref{key-lem-p}.  
\end{proof}

\begin{cor}\label{decomp-chow}
Let $\ol{X}$ be a proper $k$-scheme and $D$ be an effective Cartier divisor on it. Write $X=\ol{X}\setminus |D|$. Then:
\begin{enumerate}[{\rm (1)}]
\item
If $\chark (k)=0$, then there is a non-canonical decomposition
\[
\CH_0(\ol{X}|D)\simeq \H_0^{\Sing}(X)\oplus \operatorname{U}(\ol{X}|D).
\]
\item
If $\chark (k)=p>0$, then the canonical surjection $\pi_{\Xb,D}\colon \CH _0(\ol{X|}D)\to \H_0^{\Sing}(X)$ is an isomorphism up to $p$-torsions.
\end{enumerate}
\end{cor}

\begin{rmk}\label{rmk-Roitman-charp}
Under the much stronger assumption that $X$ is smooth and quasi-affine, Theorem \ref{div-chow} 
also follows from  Corollary \ref{cor1} for $F=h(\Xb,D)$. Indeed, in the notations of \textit{loc.~cit.}~we have $h(\Xb ,D)(k)=\CH _0(\Xb |D)$ and $\mathrm{U}(h (\Xb ,D))(k)=\mathrm{U}(\Xb |D)$.
\end{rmk}

\subsection{Discreteness of torsion $0$-cycles with modulus}
In this section, generalizing some ideas developed in \cite{LW} for the Chow group of $0$-cycles on a singular variety, we prove the useful Theorem \ref{thm-discreteness} below, showing a form of discreteness or rigidity for the torsion subgroups of the  groups $\CH_0$ with modulus. We will  frequently apply this property in the next section.

\subsubsection{}\label{LetolX}
Let $\ol{X}$ be a proper variety  over an {\it algebraically closed field} $k$ of exponential characteristic $p\geq1$. Let $D$ be an effective Cartier divisor on $\ol{X}$ and suppose that the singular locus of $\Xb$ is contained in $D$, so that the open subscheme $X=\ol{X}\setminus |D|$ is a regular (equivalently, smooth) $k$-scheme.  
We denote by $\mathit{cl}_{\ol{X}|D}$ the canonical projection morphism 
\[cl_{\ol{X}|D}\colon Z_0(X)\to \CH_0(\ol{X}| D).\]

\subsubsection{}\label{LetC}
Let $C$ be a smooth curve over $k$ and $W = \sum_{i=1}^n n_i W_i\in C_0(C\times {X}/C)$  a finite correspondence from $C$ to $\ol{X}$ such that $|{W}|\subset C\times X$.
Let $x$ be a closed point in $C$. Since $|W|$ is flat over $C$, we know that $\dim (|W|\cap (x\times X)) =0$, so that $|W|$ and $x\times X$ are in good position. Let $p_1, p_2$ be the projections from ${C}\times \ol{X}$ to $C$ and to $\ol{X}$ respectively. Then the $0$-cycle
\[W(x):=p_{2, *} (W \cap  p_1^*(x))\]
on $\ol{X}$ is well-defined and supported outside $D$.

\begin{thm}\label{thm-discreteness}
Let the notation be as in \S \S $\ref{LetolX}$ and $\ref{LetC}$. Let $n$ be an integer prime to $p$.
  Assume that there exists a dense open subset $C^o$ of $C$ such that for every $x\in C^o(k)$ one has
	\[n\cdot cl_{\ol{X}/D}(W(x)) =0.\]
	Then the function $x\in C(k)\mapsto cl_{\ol{X}/D}(W(x))$ is constant.
\end{thm}
\begin{proof}The proof uses the strategy of the proof of \cite[Proposition 4.1]{LW}. 
Let $C\subseteq \overline{C}$ be the smooth compactification of $C$.  
Let $\ol{W}_i$ be the closure of $W_i$ in $\ol{C}\times \ol{X}$ and $\ol{W}_i^N$ be its normalization, which is a smooth projective curve. Let $u_i\colon \ol{W}_i^N\to \ol{X}$ be the composite $\ol{W}_i^N\to \ol{C}\times \ol{X}\xrightarrow{p_2}\ol{X}$ (which is either a constant map into $X$ or is a finite map). By \cite[Proposition 2.10]{KP3} we have a proper push-forward map
$ u_{i,*}\colon \CH _0(\ol{W}_i^N|u_i^*(D))\to \CH _0(\ol{X}|D). $
Let $\phi _i $ be the composite $\ol{W}_i^N\to \ol{C}\times \ol{X}\xrightarrow{p_1}\ol{C}$. Set an effective Cartier divisor $D_{\ol{C},W}:= \sum _{i} \phi _{i,*}u_i^*D$ on $\ol{C}$. By \cite[Proposition 2.12]{KP3}, there is a flat pull-back map
$  \phi _i^*\colon \CH _0(\ol{C}|D_{\ol{C},W})\to \CH _0(\ol{W}_i^N|u^*_iD).  $

We define a homomorphism:
\begin{equation*}\label{WeFinallySet}
\tr_W = \sum_{i=1}^n \phi _i^*\circ u_{i,*}\colon  \CH_0(\overline{C}| D_{\overline{C}, W}) \to \CH_0(\ol{X}| D).\end{equation*}

An easy computation shows that we have for every $x\in C(k)$
	\begin{equation}\label{PropTransferChow} cl_{\ol{X}/D}(W(x)) = \tr_W(cl_{\overline{C}/D_{\overline{C}, W}}(x)).\end{equation}
Let $\CH_0(\ol{C}|D_{\ol{C}, W})^0$ denote the degree-$0$-part of the Chow group $\CH_0(\ol{C}|D_{\ol{C}, W})$. 
It is generated by differences of classes of closed points in $C^o$ by a moving argument on the curve $\ol{C}$ using the Riemann-Roch theorem, and therefore $\tr_W$ maps $\CH_0(\ol{C}|D_{\ol{C}, W})^0$ into $\CH_0(\ol{X}|D)[n]$. Note now that since $\ol{C}$ is a curve, we have $\CH_0(\ol{C}|D_{\ol{C}, W})^0 = \mathrm{Jac}(\ol{C}| D_{\ol{C}, W})(k)^0$ where the right hand side is the neutral component of the Rosenlicht-Serre generalized Jacobian. Since $n$ is prime to $\chark (k)$, $\CH_0(\ol{C}|D_{\ol{C}, W})^0$ is $n$-divisible by \cite[Chapter V]{SerreGACC}, and therefore the image of $\CH_0(\ol{C}|D_{\ol{C}, W})^0$ in $\CH_0(\ol{X}|D)$ is $0$. Hence by  \eqref{PropTransferChow}, for every pair of closed points $x_1, x_2$ in $C(k)$, we have
\begin{align*} cl_{\ol{X}/D}(W(x_1))  - cl_{\ol{X}/D}(W(x_2))  &= \tr_W(cl_{\overline{C}/D_{\overline{C}, W}}(x_1)) -\tr_W(cl_{\overline{C}/D_{\overline{C}, W}}(x_2)) \\ &= \tr_W(cl_{\overline{C}/D_{\overline{C}, W}}(x_1) -cl_{\overline{C}/D_{\overline{C}, W}}(x_2) ) = 0,\end{align*}
proving the statement.
\end{proof}

\begin{rmk}\label{rmk-disc}
(1)
Let notation be as in \S \S \ref{LetolX} and \ref{LetC}. 
The function $cl_{\ol{X}/D}(W(-))$ can be regarded as function $cl_{\ol{X}/D}(W(-))_K\colon C(K)\to \CH_0(\ol{X}_K|D_K)$ for any field extension $K/k$.
If the function $cl_{\ol{X}/D}(W(-))_K$ maps $C^o(K)$ to $\CH_0(\ol{X}_K|D_K)[n]$ for an algebraically closed field $K$, then the map is constant. 

(2)
The statement of Theorem \ref{thm-discreteness} is true for a local setting in the following sense.
Let $k$ be an algebraically closed field and let $\ol{X},D$ be as above.
Let $S$ be a semi-local scheme of a normal curve over $k$ at closed points and let $K$ be the fraction field of $S$.
Let $W \in C_0(X\times S/S)$ be a relative finite correspondence. The divisor $W$ defines as above a function
$
S(K) \to \CH_0(\ol{X}_K|D_K)
$.
If the image is contained in the $n$-torsion subgroup, then the function is constant. 
\end{rmk}

%

\begin{cor}\label{invariance-of-torsion}
Let the notations be as in \S$\ref{LetolX}$
and let $K$ be an extension
field of $k$.
Then the natural map
\[  \mathrm{CH}_0(\overline{X}|D) \lto 
\mathrm{CH}_0(\overline{X}_K|D_K) \]
is injective and induces an isomorphism 
\[  \mathrm{CH}_0(\overline{X}|D)\{ p'\} \simeq
\mathrm{CH}_0(\overline{X}_K|D_K)\{ p' \}, \] 
where $M\{ p' \} $ denotes the
prime-to-$p$-torsion subgroup of an abelian
group $M$.
\end{cor}
\begin{proof}
Suppose $z\in \mathrm{CH}_0(\overline{X}|D)$ is annihilated in $\mathrm{CH}_0(\overline{X}_K|D_K)$.
Then, by a limit argument, the relation annihilating $z_K$ is defined over $X_A$, where $A$ is a finitely generated $k$-subalgebra of $K$: i.e.~there is a flat family $\ol{C}\subset X_A$ of curves in $X$ parametrized by $\Spec {A}$ and a rational function $f\in G(\ol{C},D_A)$ with $\dv (f)= z_A$. By specializing to a $k$-rational point of $\mathrm{Spec}(A)$, we get a relation annihilating $z$; hence $z\in \mathrm{CH}_0(\overline{X}|D)$ is zero.

Having shown the injectivity, to show the surjectivity on prime-to-$p$-parts we may assume that $K$ is algebraically closed.
Suppose we are given an element $z_K\in \mathrm{CH}_0(\ol{X}_K|D_K)$ annihilated by an integer $n$ prime to the exponential characteristic. The same limit argument shows that there exist: \begin{enumerate}
\item a finitely generated smooth $k$-subalgebra $A$ of $K$;
\item a cycle $z_A$ on $X_A$ which is flat over $\mathrm{Spec}(A)$ and induces $z_K$ by the scalar extension $A\to K$;
\item a relation annihilating $n\cdot z_A$.
\end{enumerate}By specializing to an arbitrary $k$-rational point $x\colon \mathrm{Spec}(k)\to \mathrm{Spec}(A)$, we get a cycle $z$ on $X$ which is annihilated by $n$ in $\mathrm{CH}_0(\overline{X}|D)$. We show now that  $z$ maps to $z_K$ by scalar extension $k\to K$.

Consider the $K$-scheme $\mathrm{Spec}(A\otimes _k K)$. There are two distinguished $K$-rational points on it: one is $\eta \colon \mathrm{Spec}(K)\to \mathrm{Spec}(A\otimes _k K)$ which corresponds to the inclusion $A\to K$ and the other is $x\times _k K\colon \mathrm{Spec}(K) \to \mathrm{Spec}(A\otimes _k K)$. By Bertini's theorem, there is a smooth curve $C$ of $\mathrm{Spec}(A\otimes _kK)$ passing thorough $\eta$ and $x\otimes _k K$. Restriction of the flat family of cycles $z_A$ over $\Spec {A}$ by $C\to \mathrm{Spec}(A)$ gives a cycle $z_C$ on $X\times _k C=X_K\times _K C$, which is a family of $n$-torsion $0$-cycles on $X_K$ parametrized by $C$.
Then since $K$ is algebraically closed, we can apply Theorem \ref{thm-discreteness} to conclude that $z\otimes _kK=z_C(x\times _k K)$ and $z_K=z_C(\eta )$ are equal in $\mathrm{CH}_0(\overline{X}_K|D_K)$.
\end{proof}


\subsection{Torsion cycles with modulus and cycles on curves}
\subsubsection{}\label{SettingTors}
Suppose that $\overline{X}$ is a projective variety over an algebraically closed field $k$ of characteristic $p\geq 0$, regular in codimension one. 
Let $D$ be an effective Cartier divisor on $\Xb$ such that the open complement $X = \Xb\setminus |D|$ is smooth.

\def\TXD{T_{\ol{X}|D}}
\def\TXDred{T_{\ol{X}|D_{red}}}

Denote by $\TXD$ the subgroup of $\CH _0(\ol{X}|D)^0$
generated by elements $b$ for which there exists a
smooth proper curve $\ol{C}$ with a finite  morphism
$\varphi \colon \ol{C}\to\ol{X}$
satisfying $\varphi(\ol{C}) \nsubseteq D$, and an element $a\in \CH _0(\ol{C}|\varphi ^*D)_\mathrm{tors}$
such that $b=\varphi _*a$.

The main technical result of this section is Theorem \ref{tor-comes-curve}. The proof is strongly inspired by \cite[Proposition 3.4]{MarcTorsion}.
\begin{thm}\label{tor-comes-curve}
In the notation in \S \ref{SettingTors}, we have an equality $T_{\overline{X}|D}\{ p' \} = \CH_{0}(\overline{X}|D)\{p'\}. $
\end{thm}
We will actually be able to deduce from it the following more refined result that represents the heart of the proof of Theorem \ref{Torsion-curves-Intro} (see Corollary \ref{Torsion-Curves-Body} below).

\begin{prop}\label{generic-rep-T}
For any given element of $\CH _0(\ol{X}|D)\{ p'\} $, there is a smooth proper curve $\Cb$ over $k$ with a finite morphism $\varphi \colon \Cb\to \ol{X}$, such that the given element comes from $\CH _0(\Cb|\varphi ^*D)\{ p'\}$.
\end{prop}

\begin{ex}In characteristic $0$, we can show that $\CH_0(\ol{X}|D)_{\mathrm{tors}}\simeq \H_0^{\Sing}(X)_{\mathrm{tors}}$ in the case $\Xb = \Cb_1\times \Cb_2$ for $\Cb_i$ two smooth projective curves over $k$ and $D= \Cb_1\times \mathfrak{m}$ for $\mathfrak{m} = \sum_{i=1}^{r}n_i[x_i]$ an effective divisor on $\Cb_2$. The proof, that we omit, uses Proposition \ref{generic-rep-T} together with Rojtman's torsion theorem for an open subvariety of a smooth projective variety (as in the formulation of \cite{SS}).
\end{ex}

The proofs of Theorem \ref{tor-comes-curve} and Proposition \ref{generic-rep-T} require some technical works. Their eventual proofs are completed at the end of this section.

\subsubsection{}
If $\dim \ol{X}=1$, then Theorem \ref{tor-comes-curve} is trivially true. So we may assume $\dim \ol{X}\ge 2$.
The following lemma reduces the general case to the case of surfaces.

\begin{lem}
Suppose that the following equality holds for all $(\ol{X},D)$ as above whenever $\dim \ol{X}=2$:
\[T_{\overline{X}|D}\{ p' \} = \CH_{0}(\overline{X}|D)\{p'\} .\]
Then the equality holds for all $(\ol{X},D)$.
\end{lem}

\begin{proof}
 Let $z$ be a $0$-cycle whose class in $\CH_{0}(\overline{X}|D)_\mathrm{tors}$ is nonzero and $n$-torsion with $n$ prime to $p$. Then there are normal, proper, integral curves $\overline{C}_{1}, \dots, \overline{C}_{s}$ with finite morphisms $\varphi_{i}: \overline{C}_{i} \to \overline{X}$ and $f_{i} \in G(\ol{C}_i,\varphi_i^*(D))$ such that
$$nz = \sum^{s}_{i=1}\varphi_{i,*}(\dv(f_{i})).$$
One may assume that $\varphi_{i}$ maps $\ol{C}_{i}$ birationally to its image. 
By blowing up with point centers lying on $X$, one can construct a projective birational morphism $\pi \colon \ol{Y}\to \ol{X}$ such that
\begin{enumerate}
  \item
  $\pi^{-1}(X)$ is smooth and $\pi^{-1}$ is an isomorphism in a neighborhood of $D$;
  \item
  the maps $\varphi_{i}: \ol{C}_{i} \to \ol{X}$ factor through an inclusion $\phi_{i}: \ol{C}_{i} \to \ol{Y}$;
  \item
  there is a $0$-cycle $\tilde{z}$ on $\ol{Y}$, smooth projective rational curves $L_{j}$ for $j=s+1,\cdots, r$ lying in the exceptional locus, 
and rational functions $g_{j}$ on $L_{j}$ such that we have relations:                
\[\pi_{*}(\tilde{z}) = z, \quad n\tilde{z} = \sum^{s}_{i=1} \phi_{i, *}(\dv(f_{i})) + \sum^{r}_{j=s+1} \phi_{j, *}(\dv(g_{j})),\]
  where $\phi_{j}: L_{j} \to \ol{Y}$ is the inclusion.
\end{enumerate}
Furthermore, after further blow-ups, we may assume that the union
$$\ol{C} = \bigcup^{s}_{i = 1}\ol{C}_{i} \cup \bigcup^{r}_{j = s+1}L_{j}$$
has at most two components passing through any point of $\ol{C}$. In particular, $\ol{C}$ has  embedding dimension two, 
 which implies that there is a general surface section $\ol{S}$ of $\ol{Y}$ containing $\ol{C}$ which is regular in $\ol{S}\cap \pi ^{-1}(X)$ \cite[Theorems (1), (7)]{AK}. Then the assumption applied to (the normalization of) $\ol{S}$ implies that $\tilde{z} \in T_{\ol{S}|E}$. 
Composing with $\pi$, we get that $z \in \TXD$.  
\end{proof}
\subsubsection{}
From now until the end of the proof of Theorem  \ref{tor-comes-curve}, we will assume that $\dim \ol{X}=2$.
Let $n$ be a positive integer prime to $p$. Let $z$ be a $0$-cycle on $X$ of degree zero. 
Let $\ol{C}$ be a proper smooth curve in $\overline{X}$ containing $|z|$.
Then, as $\CH _0(\ol{C}|\varphi ^* D)^0$ is $n$-divisible, there is
a $0$-cycle $y$ on $\ol{C}$ with $ny = z$ in $\CH _0(\ol{C}|\varphi ^*D)^0$.

\begin{df}
Let $n$, $\ol{C}$ be as above. We define an element $ n_{\ol{C}}^{-1}(z) \in
\CH_0(\ol{X}|D)^0/\TXD $ to be the class represented by $y$.
\end{df}
This does not depend on the choice of $y$ and it satisfies $n_{\ol{C}}^{-1}(nz)=n\cdot n_{\ol{C}}^{-1}(z)=z$ in
$\CH _0(\ol{X}|D)^0/\TXD$. The next proposition shows that it also does not depend on the choice of the curve $\ol{C}$.
\begin{prop}\label{indep-of-C}
Let $z$ be a $0$-cycle on $X$ of degree zero and $\ol{C},\ol{C}'$ be two curves satisfying the above conditions with
respect to $z$.
Then we have an equality 
$ n_{\ol{C}}^{-1}(z)=n_{\ol{C}'}^{-1}(z) $
in $\CH_0(\ol{X}|D)^0/\TXD$.
\end{prop}
\begin{proof}

Fix a closed point $x_0\in X$, $x_0 \notin |z|$. Let $ P=\{ \ol{C}_x=H_x\cdot \ol{X} ~|~ x\in \mathbb{P}^1 \}
$ be a pencil of hypersurface sections of $\Xb$
satisfying the following conditions
\begin{enumerate}
\item The generic member $\ol{C}_t$ of $P$ is smooth and irreducible, and misses the singular locus of $D_{\mathrm{red}}$.
\item The base locus of $P$ contains $|z| \cup \{ x_0 \} $ and misses $D$. 
\item The equality of Cartier divisors $\ol{C}_0=\ol{C}+E$ and $\ol{C}_\infty =\ol{C}'+E'$ holds on $\ol{X}$,
where $E$ and $E'$ are smooth and
irreducible, intersect $D$ at least in one point each,
but miss the singular locus of $D_{\mathrm{red}}$. Moreover they are disjoint from the base locus of $P$.
In addition, $\ol{C}_0$ and $\ol{C}_\infty$ have only ordinary double points as singularities.
\end{enumerate}
(The condition that $E$ and $E'$ meet $D$ is automatic if the hypersurfaces have sufficiently high degree.)
 
By blowing up along the base locus of $P$ we get a morphism $\pi_P\colon \Xb_P \to \mathbb{P}^1$. We denote by $u$ the blowing down map $u\colon \Xb_P\to \Xb$. Set $X_P= u^{-1}(X)$ and $D_P = u^*D$. Write $z=\sum _i n_i p_i$ and set $Z=\sum n_i u^{-1}(p_i)$ which is a divisor on $X_P$ and satisfying
$ \ol{C}_x\cdot Z=z $ 
on $\ol{C}_x $
for every member $\ol{C}_x$ of the pencil $P$. 

Let $S$ be the spectrum of the local ring of $\mathbb{P}^1$ at $0$, and denote by $s$ its closed point, by $\eta$ its generic point and by $\etab$ a geometric generic point. We denote by $\Xb_S\to S$ the base-change of $\pi_P$ to $S$: it is a semistable projective curve over $S$. By construction, the special fiber $(\Xb_S)_s$ coincides with $\ol{C}_0 = \ol{C} + E$, while the generic fiber $\ol{C}_{\eta}:=(\Xb_S)_\eta $ represents the generic member of the pencil. 

The choice of the extra point $x_0$ outside  $D$ determines a section $s_0:\mathbb{P}^1\to \ol{X}_P$ of $\pi_P$. We let $s_0$ denote the closed subscheme of $\ol{X}_P$ given by it as well.

Consider the presheaf of abelian groups $\mathbf{Pic}^0_{(\ol{X}_S|D_S\sqcup s_0)/S}$ on the category of separated schemes of finite type over $S$ given by 
$T\mapsto \{ $pairs $(\mathcal{L},\alpha ) \}$ where $\mathcal{L}$ is a line bundle on $X_S\times _S T$ which has degree zero along every fiber over $T$ and $\alpha $ is an isomorphism $\mathcal{L}|_{(D_S \sqcup s_0)\times _ST}\simeq \mathcal{O}_{(D_S \sqcup s_0)\times _ST}$.
It is representable by a scheme locally of finite type over $S$ (cf.~Lemma \ref{RepresRelPic} below). 
The divisor $Z\subset {X}_P$ determines a section
$Z\colon S\to \mathbf{Pic}^0_{(\ol{X}_S|D_S\sqcup s_0)/S}$.
Take a point $\xi_s$ (resp.\ $\xi_{\etab}$) on the closed fiber $\mathrm{Jac}{(\ol{C}_0|(D\sqcup s_0)\cdot \ol{C}_0)}$ (resp.~on the geometric generic fiber $\mathrm{Jac}{(\ol{C}_{\etab}|(D\sqcup s_0)\cdot \ol{C}_{\etab})}$) of $\mathbf{Pic}^0_{(\ol{X}_S|D_S\sqcup s_0)/S}$ such that 
$n\xi_s=Z_s $ (resp. $n\xi_{\etab}=Z_{\etab}$ )
and that $\xi_s$ is a specialization of $\xi_{\etab}$.
Here we used the $n$-divisibility of $\mathrm{Jac}{(\ol{C}_0|(D\sqcup s_0)\cdot \ol{C}_0)}$ and $\mathrm{Jac}{(\ol{C}_{\etab}|(D\sqcup s_0)\cdot \ol{C}_{\etab})}$.
Then there is a spectrum $S'$ of a DVR dominating $S$ and a morphism 
%
\[
\gamma': S' \lto \mathbf{Pic}^0_{(\ol{X}_S|D_S\sqcup s_0)/S}
\]
such that $\gamma'(s')=\xi_s$ and $\gamma'(\etab')=\xi_{\etab}$. 
Here $s'$ and $\etab'$ are the closed point and the geometric generic point of $S'$.
There is a Cartier divisor $W$ on $\ol{X}_S\times_SS'$ finite flat over $S'$ representing $\gamma'$. It naturally gives an element in $C_0(\ol{X}_{S'}/S')$.
Then we have $n \cdot cl(W(s'))=u_*(Z_s)$ in $\CH _0(\ol{X}_{s'}|D_{s'})$ and $n\cdot cl(W(\ol{\eta}'))=u_*(Z_{\ol{\eta }})$ in $\CH _0(\ol{X}_{\ol{\eta}'}|D_{\ol{\eta}'})$.

Then the image of the map $\phi=cl({W}(-))-cl({W}(s')):S'(k(\ol{\eta }'))\lto\CH_0(\ol{X}_{\ol{\eta}'}|D_{\ol{\eta}'})$ lies in the $n$-torsion subgroup of the target, since $u_*(Z_s)=u_*(Z_{\etab})$ in $\CH_0(\ol{X}_{\etab}|D_{\etab})$.
By the discreteness Theorem \ref{thm-discreteness} (in the formulation of Remark \ref{rmk-disc}(2)) and by $\phi(s')=0$, the map $\phi$ is identically zero.
Therefore we have
\[ 0 =cl({W}(\etab'))-cl({W}(s'))
=n_{\ol{C}_{\etab}}^{-1}(z)-n_{\ol{C}}^{-1}(z).
\]
Hence $n_{\ol{C}_{\etab}}^{-1}(z)=n_{\ol{C}}^{-1}(z)$.
Similarly $n_{\ol{C}_{\etab}}^{-1}(z)=n_{\ol{C}'}^{-1}(z)$.
This completes the proof of Proposition \ref{indep-of-C}.
\end{proof}

In the above proof, we used the following lemma.
We let $S$ be the spectrum of a discrete valuation ring and denote by $s$ the closed point of $S$ and by $\eta$ the generic point of $S$.

\begin{lem}\label{RepresRelPic} Let $\pi \colon X \to S$ be a semistable projective curve over $S$, 
i.e.~$\pi$ is a projective and flat morphism of relative dimension $1$ whose geometric fibers are reduced, connected curves having only ordinary singularities. Assume $\pi$ is smooth over $\eta$ and admits a section $s_0$. Let $D$ be an effective Cartier divisor on $X$ which is flat over $S$. Then the relative 
Picard functor $\mathbf{Pic}^0_{(X|D\sqcup s_0)/S}$ 
is representable by a scheme (locally) of finite type over $S$.	\end{lem}
	\begin{proof}
The presheaf $\mathbf{Pic}^0_{(X| s_0)/S}$ of line bundles of degree $0$ with a fixed trivialization along $s_0$ is isomorphic to $P^0$ in \cite[\S \S (1.2) and (3.2)d, cf.~\S (1.3)]{Raynaud}.
Thanks to the semi-stability assumption, the conditions \cite[(N)${}^*$(6.1.4)]{Raynaud} and \cite[Theorem 8.2.1 (i)]{Raynaud} are satisfied in our situation. Therefore by [{\it op.~cit.}, implication (i)$\Rightarrow $(vi)], the presheaf $\mathbf{Pic}_{(X|s_0)/S}$ is representable by a scheme (locally) of finite type over $S$. 
		 Forgetting the extra trivialization along $D$ gives a canonical morphism of sheaves 
		\[\phi\colon \mathbf{Pic}^0_{(X|D\sqcup s_0)/S} \to \mathbf{Pic}^0_{(X|s_0)/S}. \]
		Let $G$ be the affine group scheme
 over $S$ given by $\pi_{D, *} \mathbb{G}_{m, D}$. Then it is easy to show that $\phi$ makes $ \mathbf{Pic}^0_{(X|D\sqcup s_0)/S}$ a $G$-torsor over $\mathbf{Pic}^0_{(X|s_0)/S}$. Hence, by the relative representability theorem \cite[Lemma 3.6]{Gro}, the presheaf $\mathbf{Pic}^0_{(X|D\sqcup s_0)/S}$ is representable by a scheme locally of finite type over $S$.
\end{proof}

Let $Z_0(X)^0$ denote the group of $0$-cycles on $X$ of degree zero.
By Proposition \ref{indep-of-C} we have a well-defined homomorphism
\[ n_X^{-1}: Z_0(X)^0\to \CH _0(\ol{X}|D)^0/\TXD \]
satisfying $n_X^{-1}(nz)=n\cdot n_X^{-1}(z)=z$.
In order to complete the proof of Theorem \ref{tor-comes-curve}, we need two more lemmas.
\begin{lem}\label{blow-up-lem}
Let $u: \ol{X}'\to \ol{X}$ be a blow-up at a point $p$ in $X$. Then the following diagram commutes:
\[ \xymatrix{
Z_0(X')^0\ar[r]^(0.4){n_{X'}^{-1}} \ar[d]_{u_*}&\CH _0(\ol{X}'|D)^0/\TXD \ar[d]_{u_*} \\%
Z_0(X)^0\ar[r]^(0.4){n_{X}^{-1}}&\CH _0(\ol{X}|D)^0/\TXD 
} \]
\end{lem}
\begin{proof}
Write an element $z'\in Z_0(X')^0$ as $z'=z'_1+z'_2$ where
$z'_1$ is supported on the exceptional divisor $E$
of $u$ and $z'_2$ is supported on $X'\setminus E$.
Set $d=\deg z'_1$. Choose $q\in E$ and write
\[ z'=(z'_1-d\cdot q)+(d\cdot q+z'_2) \]
as sum of $0$-cycles of degree zero.
The first term vanishes when we apply $u_*$ and when we apply
$u_*\circ n_{X'}^{-1}$, so we may assume $z'$
is of the form $z'=d\cdot q +z'_2$.
Take a proper smooth curve in $X$ which contains
$|u_*z'|$ and passes through $p$ having the right tangent
direction so that the strict transform $C'\subset X'$
(which is isomorphic to $C$) contains $q$.
We have a tautological identity
\[ u_*(n_{C'}^{-1}(z'))=n_C^{-1}(u_*z'). \]
The left hand side is equal to $u_*(n_{X'}^{-1}(z'))$,
and the right hand side to $n_X^{-1}(u_*z')$.
\end{proof}

\begin{lem}\label{Prop-Factn_-1}
The map $n_X^{-1}$ factors through $\CH _0(\ol{X}|D)^0/\TXD$.
\end{lem}
\begin{proof}
It suffices to show the following:
let $\varphi \colon \ol{C}\to \ol{X}$ be a morphism from a proper smooth curve such that $\varphi $ is a birational map to its image and such that $\varphi(\ol{C})\not \subset D$,
and let $z$ be a $0$-cycle on $\ol{C}$ which represents
a torsion class in $\CH _0(\ol{C}|\varphi ^*D)^0$.
Then $n_X^{-1}(j_*z)=0$. To achieve this, 
blow up $X$ at  the singular points of $\varphi (C)$ 
several times so that
the strict transform of $\varphi (C)$ is non-singular,
therefore isomorphic to $C$.
If we take a $0$-cycle $y$ on $\ol{C}$ such that 
$ny=z$ in $\CH _0(\ol{C}|\varphi ^*D)^0$, by Lemma \ref{blow-up-lem} applied to $z\in Z_0(X')^0$ we have
\[ n_X^{-1}(\varphi _*z)=\varphi _*y \text{ in }\CH _0(\ol{X}|D)^0/\TXD. \]
Since $y$ is a torsion class on $\ol{C}$, the right hand side is zero.
This completes the proof.
\end{proof}
\begin{proof}[{\bf Proof of Theorem \ref{tor-comes-curve}}]
By Lemma \ref{Prop-Factn_-1}, we have a map
\[ n_X^{-1}: \CH _0(\ol{X}|D)^0/\TXD\to \CH _0(\ol{X}|D)^0/\TXD \]
such that for every $z\in \CH _0(\ol{X}|D)^0/\TXD$ we have
$n_X^{-1}(nz)=n\cdot n_X^{-1}(z)=z$.
If $z\in \CH _0(\ol{X}|D)^0$ is annihilated by an integer $n$ prime to $p$, then we have $z=n_X^{-1}(nz)=0$ in $\CH _0(\ol{X}|D)^0/\TXD$.
This completes the proof of Theorem \ref{tor-comes-curve}. 
\end{proof}

\begin{lem}\label{exist-pencil}
Let $\ol{C}$ be a smooth proper curve and $\varphi \colon \ol{C} \to \ol{X}$ be a morphism with $\ol{C}\to \varphi (\ol{C})$ birational.
Suppose we are given a pencil
$P=\{ \ol{C}_x=H_x\cdot \ol{X}~|~x\in \mathbb{P}^1  \}$ satisfying:

\begin{enumerate}
\item $\ol{C}_1=\varphi (\ol{C})+E$, where $E$ is a smooth proper irreducible curve missing $\varphi (\ol{C})\cap |D|$ such that
$E$ intersects $\varphi (\ol{C})$ transversally.
\item The axis misses $D$ and intersects $X$ transversally.
\end{enumerate}

Then we have:
\[ \varphi _*(\mathrm{CH}_0(\ol{C}|\varphi ^*D)\{ p'\} )\subset 
\varphi _{\ol{C}_{\etab}*} (\mathrm{CH}_0(\ol{C}_{\etab}|D\cdot \ol{C}_{\etab})\{ p' \} ) \text{ in }\CH _0(\ol{X}_{\etab}|D_{\etab})_\mathrm{tors}, \]
where $\ol{C}_{\etab }$ is the geometric generic member of the pencil.

Moreover, the map
\[ \varphi _{C_{\etab}*}\colon\mathrm{CH}_0(\ol{C}_{\etab}|~D\cdot \ol{C}_{\etab})\{ p'\} \to 
\CH _0(\ol{X}_{\etab}|D_{\etab})\{ p'\}
\overset{\mathrm{Cor.~}\ref{invariance-of-torsion}}{=}
\CH _0(\ol{X}|D)\{ p'\}  \] is $\mathrm{Gal}(\etab/\eta)$-equivariant.
\end{lem}

\begin{proof}
By blowing along the base locus of $P$
we get a morphism
$\pi _P\colon \ol{X}_P\to \mathbb{P}^1$. Since the base locus misses $D$, we can view $D = D_P$ as a Cartier divisor on $\ol{X}_P$.
If we choose a point $x_0$ in the base locus, we get a section $s_0\colon\mathbb{P}^1 \to \ol{X}_P$.
Denote by $S$ the local scheme of $\mathbb{P}^1$ at $1$.
Let $\ol{X}_S$ and $D_S$ be the base change to $S$ of $\ol{X}_P$ and $D_P$ respectively. 
The Cartier divisor $D_S$ on $\ol{X}_S$ is finite and flat over $S$ 
 and  we can use Lemma \ref{RepresRelPic} to show that $\mathbf{Pic}^0_{({\ol{X}_S}|D_S\sqcup s_0)/S}$ is representable.
Take any $y\in \mathrm{Jac}(\ol{C}|\varphi ^*D)\{ p'\}$ and
a lifting $y'$ of $(y,0)\in \mathrm{Jac}(\ol{C}|\varphi ^*D)\oplus \mathrm{Jac}(E|D\cdot E)$ in the surjection
\[ \mathrm{Jac}(\ol{C}_1|(D\cdot \ol{C}_1)\sqcup x_0)\{ p'\} \twoheadrightarrow
\mathrm{Jac}(\ol{C}|\varphi ^*D)\{ p'\} \oplus \mathrm{Jac}(E|D\cdot E)\{ p'\} . \]
Then the image of $y$ in $\CH _0(\ol{X}|D)$ is equal to the image of $y'$ by the composite map
\[ \mathrm{Jac}(C_1|(D\cdot C_1)\sqcup x_0)\{ p'\} \twoheadrightarrow
\mathrm{Jac}(\ol{C}|\varphi ^*D)\{ p'\} \oplus \mathrm{Jac}(E|D\cdot E)\{ p'\} \to \CH _0(\ol{X}|D).  \]
Suppose that the lift $y'$ 
is annihilated by an integer $n$ prime to $p$.
Then there is an irreducible component $B$ of 
$\mathbf{Pic}(({\ol{X}_S}|D\sqcup s_0)/S)[n]$ containing 
$y'$, since this group scheme is \'etale over $S$. Note that $B$ is regular and $1$-dimensional.
There is a horizontal Cartier divisor $W$ on $\ol{X}_S\times _S B$ which represents the section $B\to \mathbf{Pic}((\ol{X}_S|D\sqcup s_0)/S)$.
Then by the rigidity (Remark \ref{rmk-disc}) of
\[ cl(W(-))\colon B(k(\ol{\eta}))\to \coprod _{b\in B(k(\ol{\eta}))} \mathrm{Jac}(\ol{C}_b|~D\cdot \ol{C}_b) [n]\longrightarrow  \CH _0(\ol{X}|D)[n] \]
we find that $y'$ is in the image of $ \mathrm{Jac}(\ol{C}_{\ol{\eta}}|~D\cdot \ol{C}_{\ol{\eta}})[n]$.

For the second assertion, we first take any $n$ prime to $p$
and
$z\in 
\mathrm{Jac}(\ol{C}_{\etab }|D\cdot \ol{C}_{\etab }\sqcup x_0)({\etab })[n]$.
Let $B$ be the closure in 
$\mathbf{Pic}_{(\ol{X}_P|D\sqcup s_0)/S}[n]$
of its image in 
$\mathrm{Jac}(C_{\eta}|D\cdot C_{\eta}\sqcup x_0)[n]$. 
If we let
$\sigma \in \mathrm{Gal}(\etab /\eta )$
act on $z$, the resulting element
\[ z^\sigma :\mathrm{Spec}(k(\etab ))
\xrightarrow{\sigma} \mathrm{Spec}(k(\etab ))
\xrightarrow{z} B\]
lands on the same point of $B$.
On the other hand, by the discreteness Theorem \ref{thm-discreteness},
$B(\overline{\eta})\to \CH _0(\ol{X}|D)[n]$
is a constant map.
Therefore $z$ and $z^\sigma$ maps to the same element 
of $\CH _0(\ol{X}|D)[n]$. 
This shows that the map
\[ \mathrm{Jac}(\ol{C}_{\etab}|D\cdot \ol{C}_{\etab}\sqcup x_0)({\etab})\{ p'\}\to \CH _0(\ol{X}|D) \]
is $\mathrm{Gal}(\etab /\eta )$-equivariant. It follows that so is the map
\[ \mathrm{Jac}(\ol{C}_{\etab}|D\cdot C_{\etab})({\etab})\{ p'\}\to \CH _0(\ol{X}|D). \]
This complets the proof of Lemma \ref{exist-pencil}.
\end{proof}

\begin{proof}[{\bf Proof of Proposition \ref{generic-rep-T}}]
Take any two elements $y_1,y_2\in \CH _0(\ol{X}|D)\{ p' \}$, which comes from torsion parts of $\CH _0$ of smooth proper curves $ \ol{C}^{(1)}, \ol{C}^{(2)}$ respectively.
We show that there is a smooth proper curve $\ol{C}^{(3)}$ mapping to $\ol{X}$ and an element in $\CH _0( \ol{C}^{(3)}|D\cdot  \ol{C}^{(3)})\{ p' \}$ which maps to $y_1+y_2$.
We may assume for $i=1,2$, that $\ol{C}^{(i)}\to X$ maps $\ol{C}^{(i)}$ birationally to the image.
Take a pencil such that 
\begin{enumerate}
\item $\ol{C}_0=\varphi ^{(1)}(\ol{C}^{(1)})+E^{(1)}$, where $E^{(1)}$ is a smooth proper irreducible curve missing $\varphi ^{(1)}(\ol{C}^{(1)})\cap |D|$ such that
$E^{(1)}$ intersects $\varphi ^{(1)}(\ol{C}^{(1)})$ transversally.
\item $\ol{C}_1=\varphi ^{(2)}(\ol{C}^{(2)})+E^{(2)}$, where $E^{(2)}$ is a smooth proper irreducible curve missing $\varphi ^{(2)}(\ol{C}^{(2)})\cap |D|$ such that
$E^{(2)}$ intersects $\varphi ^{(2)}(\ol{C}^{(2)})$ transversally.
\item The axis misses $D$ and intersects $X$ transversally.
\end{enumerate}

Then we can apply Lemma \ref{exist-pencil}
to deduce that the image of
$\mathrm{Jac}(\ol{C}_{\etab}|~D\cdot \ol{C}_{\etab})\{ p' \} $
contains the images of
$\mathrm{Jac}(\ol{C}^{(1)}|\varphi ^{(1)*}D)\{ p' \}$ and $\mathrm{Jac}(\ol{C}^{(2)}|\varphi ^{(2)*}D)\{ p' \} $.
Specializing to a generic member of the pencil, 
we find that the given elements $y_1,y_2$ come from $\ol{C}^{(3)}:=\ol{C}_x$ for some $x\in \mathbb{P}^1(k)$.
Repeating this argument finitely many times, we find that any element of $T$ comes from an appropriate smooth proper curve $\ol{C}$ over $k$.
\end{proof}

\begin{cor}[see Theorem \ref{Torsion-curves-Intro}]\label{Torsion-Curves-Body}
Let $k$ be an algebraically closed field of exponential characteristic $p\geq 1$.
Let $\Xb$ be a projective variety over $k$, regular  in codimension one. Let $D$ be an effective Cartier divisor on $\Xb$ such that the open complement $X= \Xb \setminus |D|$ is smooth over $k$. Let $\alpha\in \CH_0(\Xb|D)$ be a prime-to-$p$-torsion cycle. Then there exist a smooth projective curve $\Cb$ with a morphism $\varphi \colon \Cb \to \Xb$ such that $\varphi ^{*}D$ is a well defined Cartier divisor on $\Cb$ and a prime-to-$p$-torsion $0$-cycle $\beta \in \CH_0(\Cb| (\varphi ^*D)_{\mathrm{red}}) \{ p'\} $ such that $\varphi _*(\beta) = \alpha $ in $\CH _0(\overline{X}|D)$.
\end{cor}
\begin{proof}
By Proposition \ref{generic-rep-T}, it is enough to show that for a proper smooth curve $\Cb$ and an effective divisor $D$ on $\Cb$, we have the following isomorphism
\[
\pi_{\Cb,D}\{p'\}\colon \CH_0(\Cb|D)\{p'\}\xrightarrow{\simeq} \CH_0(\Cb|D_\mathrm{red})\{p'\}= \H_0^{\Sing}(C)\{p'\}.
\]
The second equality is true because $\ol{C}$ is a curve.
The first map is an isomorphism because its kernel $\operatorname{U}(\ol{C}|D)=G(\ol{C},D_\mathrm{red})/G(\ol{C},D)$ is uniquely $n$-divisible for any $n$ prime to $p$ by Lemma \ref{key-lem} and Lemma \ref{key-lem-p}.
\end{proof}
\begin{rmk}For a more definitive result regarding the torsion of $\CH_0(\Xb|D)$ in characteristic zero, 
using a completely different method, see \cite[Theorem 7.8 and Theorem 8.8]{BK}.
    \end{rmk}


\section{Reciprocity presheaves and sheaves}\label{Section-Reciprocity}

\subsection{Recall of definitions and fundamental results}
We denote by $\Sch /k$ the category of separated schemes of finite type over $k$ and by $\Sm/k$ the subcategory of smooth schemes over $k$. Let $\mathbf{Cor}/k$ be the category of finite correspondences over $k$: it has the same class of objects of $\Sm /k$, and the set of morphisms from $X$ to $Y$ is
$\mathrm{Hom}_{\mathbf{Cor}/k}(X,Y)=\mathrm{Cor}(X,Y):=C_0(X\times Y/X)$. A {\it presheaf with transfers} is a presheaf of abelian groups on $\mathbf{Cor}/k$ (see \cite[Lecture 2]{MVW} for their basic properties).
We write $\PST$ for the abelian category of presheaves with transfers.

The following definitions are taken from \cite[\S 2]{KSY}.

\begin{df}\label{def-mod-pair}
A pair $(\ol{X}, Y)$ of schemes is called \emph{a modulus pair} if
\begin{romanlist}
\item
$\ol{X} \in \Sch/k$ is integral and proper over $k$;
\item
$Y \subset \ol{X}$ is a closed subscheme such that $X = \ol{X} \setminus Y$ is quasi-affine (i.e.~quasi-compact and isomorphic to an open subscheme of an affine scheme) and smooth over $k$.
\end{romanlist}
\end{df}
Let $(\ol{X}, Y)$ be a modulus pair and write $X = \ol{X} \setminus Y$ for the quasi-affine open complement. For $S \in \Sm/k$, we denote by $\Cc_{(\ol{X}, Y)}(S)$ the class of morphisms $\varphi\colon \ol{C} \to \ol{X} \times S$ fitting in the following commutative diagram
\[\xymatrix{ & \ol{C} \ar[d]^{\varphi} \ar[rd]^{\gamma_\varphi} \ar[ld]_{p_\varphi} & \\
 S & \Xb\times S\ar[r] \ar[l]  & \Xb, 
}\]
where $\ol{C} \in \Sch/k$ is an integral normal scheme and $\varphi$ is a finite morphism such that, for some generic point $\eta$ of $S$, $\dim \ol{C} \times_{S} \eta = 1$ and the image of $\gamma_{\varphi}$ is not contained in $Y$.

Let $G(\ol{C}, \gamma_{\varphi}^{*}Y)$ as in \S\ref{DefChowMod}. Then the divisor map on $\ol{C}$ induces
$$\mathrm{div}_{\ol{C}}: G(\ol{C}, \gamma_{\varphi}^{*}Y) \to C_0(C/S),$$
where $C = \ol{C} \setminus \gamma^{*}_{\varphi}Y$.

\begin{df}
Let $F\in \PST$ be a presheaf with transfers, $(\ol{X}, Y)$ a modulus pair with $X = \ol{X} \setminus Y$. Let $a \in F(X)$.  We say that  $Y$ is a modulus for $a$ if for every $\varphi: \ol{C} \to \ol{X} \times S \in \Cc_{(\ol{X}, Y)}(S)$ and every $f \in G(\ol{C}, \gamma_{\varphi}^{*}Y)$, we have
\[(\varphi_{*}\mathrm{div}_{\ol{C}}(f))^{*}(a) = 0\]
in $F(S)$.
Here $\varphi_{*}\colon C_0(C/S) \to C_0(X \times S/S)=\mathrm{Cor}(S,X)$ denotes the push-forward of correspondences. 
\end{df}

\begin{df}
Let $F\in \PST$ be a presheaf with transfers. We say $F$ has \emph{reciprocity} (or that $F$ is a \emph{reciprocity presheaf}) if, for any quasi-affine $X \in \Sm /k$, any $a \in F(X)$, and any open immersion $X \hookrightarrow \ol{X}$ with $\ol{X}$ integral proper over $k$, $a$ has modulus $Y$ for some closed subscheme $Y \subset \ol{X}$ such that $X = \ol{X} \setminus Y$. Following the notation in \cite{KSY}, we use $\REC$ to denote the full subcategory of the category of presheaves with transfers consisting of reciprocity presheaves. Note that $\REC$ is an abelian category.
\end{df}

\begin{rmk} \label{Phi-def}Let  $(\ol{X}, Y)$ be a modulus pair. Denote the category of abelian groups by $\Ab$. By \cite[Theorem 2.1.5]{KSY}, the functor
\begin{equation*}\label{FunctHavModulus}\PST \to \Ab, \quad F\mapsto \{a \in F(X) |  a \text{ has modulus }Y\} \end{equation*}
is representable by a presheaf with transfers, denoted by $h(\ol{X}, Y)$. 
It is explicitly constructed by 
\[ S\mapsto \operatorname{coker}\left( \bigoplus _{\varphi \in \mathcal{C}_{(\ol{X},Y)}(S) } G(\ol{C},\gamma _{\varphi }^*Y)\xrightarrow{\varphi _*\circ \dv } C_0(X\times S/S) \right) .\]
If $Y$ happens to be an effective Cartier divisor on $\ol{X}$, then $h(\ol{X}, Y)\in \REC$. 
\end{rmk}

\subsection{Homotopy invariance and torsion reciprocity sheaves}\label{section-htpy-rec}
\begin{thm}\label{htpy}
Let $F$ be a reciprocity presheaf with transfers which is separated for the Zariski topology.
Then $F$ is homotopy invariant in the following cases:
\begin{enumerate}[{\rm (1)}]
\item
$\chark (k)=0$ and $F$ is torsion;

\item
$\chark (k)=p>0$ and $F$ is $p$-torsion free.

\end{enumerate}
\end{thm}

\begin{proof}
We first prove $(1)$.
By \cite[Theorem 3.1.1(2)]{KSY}, it suffices to show that any element $a \in F(X)$ has reduced modulus. 
Since $F$ is separated for the Zariski topology, we can use the criterion given by  \cite[Remark 4.1.2]{KSY} (which depends on Injectivity Theorem \cite[Theorem 6]{KSY}). 
Namely, let $K=k(S)$ be the function field of a connected $S\in \Sm /k$ and $\ol{C}$ be a normal integral proper curve over $K$. Let $\varphi \colon \ol{C}\to \ol{X}\times _kK$ be a finite morphism such that $\varphi (\ol{C})\not\subset Y\times _kK$. Put $C=\varphi ^{-1}(X\times _kK)$ and let $\psi \colon C\to X$ be the induced map. Let $D=\varphi ^{-1}(Y\times _kK)$. Then the element $a$ has reduced modulus if, for all $\varphi \colon \ol{C}\to \ol{X}\times _kK$ as above, the map
\[ (\psi _*\dv (-))^*(a) \colon G(\ol{C},D_{\mathrm{red}})\to F(K) \]
is zero.

Since $F$ is torsion, there is an integer $n>0$ such that $na=0$.
Thus the above map factors through $G(\ol{C},D_\mathrm{red})/n$.
Since $F$ is a reciprocity presheaf with transfers, it has in particular weak reciprocity in the sense of \cite[Definition 5.1.6]{KSY}. By definition, there exists then an effective divisor $E$ on $\ol{C}$ which is a weak modulus for $\psi^*(a)$, with $|E|= |D|$.
By Lemma \ref{key-lem}, we have
$G(\ol{C},E)/n\simeq G(\ol{C},D_\mathrm{red})/n,
$
so that  the map
$(\psi _*\dv (-))^*(a)\colon G(\ol{C},D_\mathrm{red}) \to G(\ol{C},D_\mathrm{red})/n\simeq G(\ol{C},E)/n\to F(K)
$
is zero.
This proves (1).

We now prove the case (2). Again, it suffices to show that any element $a \in F(X)$ has reduced modulus. We use the following variant of \cite[Remark 4.1.2]{KSY} deduced from \cite[Theorem 6]{KSY}:
the element $a$ has  reduced modulus if for any morphism $\varphi \colon \overline{C}\to \ol{X}$ as above and for any $f\in G(\ol{C},D_{\mathrm{red}})$, the element
$(\varphi _*\dv (f))^*(a) \in F(K)$
is zero.

Now, since $F$ is a reciprocity presheaf, the section $a$ has a modulus $Y\subset \overline{X}$ supported on 
$\overline{X}\setminus X$.
By Lemma \ref{key-lem-p},
for a large $n>0$, we have $f^{p^n}\in G(\overline{C},\varphi ^* (Y\times _kK))$.
Since $Y$ is a modulus for $a$, we have
$ (\varphi _* \mathrm{div}(f^{p^n}))^*(a)=0 \text{ in } F(S), $
that is,
\[ p^n (\varphi _* \mathrm{div}_{\overline{C}}(f))^*(a)=0 \text{ in } F(S). \]
But by assumption $F(S)$ is a $p$-torsion free abelian group, so that
$ (\varphi _* \mathrm{div}_{\overline{C}}(f))^*(a)=0. $
This completes the proof.
\end{proof}


\subsection{Unipotency and divisibility}\label{unip-section}

Recall that, by Chevalley's Theorem, every algebraic group $G$ over a perfect field $k$ can be written  as an extention of a semi-abelian variety $A$ by a unipotent group $U$
\begin{equation}\label{UtoGtoA}
 0\to U \to G \to A \to 0,
\end{equation}
where \eqref{UtoGtoA} is exact when $U$, $G$ and $A$ are considered as sheaves for the \'etale (or the Zariski) topology. 
Note that every commutative algebraic group over $k$ defines a presheaf with transfers, cf.~\cite[Proof of Lemma 3.2]{SS}.

For the rest of the section, by a sheaf we will mean a Zariski sheaf on $\Sm /k$. If $F$ is a sheaf with transfers, we denote by  $\H_i(F)$  the $i$-th homology sheaf of the Suslin complex $C_*(F)$ of $F$. 
This is defined as follows (see \cite[Definition 1.4]{MVW}): using the cosimplicial scheme $\{ \Delta ^i  \} _{i\ge 0}$ with $\Delta ^i=\Spec {k[x_0,\dots ,x_i]/(x_0+\dots +x_i-1)}$, the $i$-th term $C_i(F)$ of the Suslin complex is given by $C_i(F)=F(-\times \Delta ^i)$. The differentials are given by alternating sums of the face maps.
It is known that $\H_i(F)$ are homotopy invariant for every $i\geq 0$ \cite[Corollary 2.19]{MVW}.

\begin{prop}
For every unipotent group U, we have $\H_0(U)=0$.
\end{prop}
\begin{proof}
We have to show that for every smooth $k$-scheme  $X$, the map of abelian groups
\[ i_0^* - i_1^* : U(X\times \mathbb{A}^1) \to U(X) \]
is surjective. Note that $U$ is isomorphic to an affine space $\mathbb{A}^n$ as a scheme.
Fix an isomorphism $U\cong \mathbb{A}^n$ mapping $0\in U$ to the origin.
Translating the ``multiplication by $t\in \mathbb{A}^1$'' map by this isomorphism,
we have a morphism of schemes  $\mu\colon U\times \mathbb{A}^1 \to U$, 
which coincides with $\mathrm{id}_U$ on $U\times \{ 1 \}$ and with the constant map to $0$
on $U\times \{ 0 \}$.
Now given an $f\in U(X)$, we define a section $\tilde{f}\in U(X\times \mathbb{A}^1)$ by the
composition
\[ X\times \mathbb{A}^1 \xrightarrow{f\times \mathrm{id} } U\times \mathbb{A}^1
\xrightarrow{\mu} U \]
Then we clearly have $i_0^* (\tilde{f})= 0$ and $i_1^*(\tilde{f})=f$,
so the section $- \tilde{f} \in U(X\times \mathbb{A}^1)$ does the job.
\end{proof}

\begin{cor}
Let $G$ be an algebraic group, which is an extention of a semi-abelian variety $A$
by a unipotent group $U$ as in $(\ref{UtoGtoA})$.
Then we have $\H_0(G)=A$. In particular, an algebraic group $G$ over a perfect field is unipotent
if and only if one has $\H_0(G)=0.$
\end{cor}

This corollary motivates the following definition.

\begin{df}\label{def-uf}
\begin{enumerate}[$(1)$]
\item
A reciprocity Zariski sheaf $F$ is said to be {\it unipotent} if it satisfies $\H_0(F)=0$.
\item
For a reciprocity sheaf $F$, we define a reciprocity sheaf $\operatorname{U}(F)$
 (the {\it unipotent part} of $F$)
to be the kernel of the canonical surjection $F\to \H _0(F)$
\[ \operatorname{U}(F)= \ker (F \to \H_0(F)). \]
\end{enumerate}
(note that the definitions themselves make sense for any abelian Zariski sheaf).
\end{df}

\begin{rmk}
\begin{enumerate}
\item
An algebraic group over a perfect field $k$ is unipotent as a reciprocity sheaf if and only if
it is unipotent in the classical sense.
\item
The unipotent part of a sheaf $F$ is unipotent in the sense of Definition \ref{def-uf}; this follows from the long exact sequence of Suslin homology arising from the short exact sequence $0\to \operatorname{U}(F)\to F\to \H _0(F)\to 0$:
\[ 0=\H_1(\H_0(F))\to \H_0(\operatorname{U}(F)) \to \H_0(F)\xrightarrow{\cong} \H_0(\H_0(F)).  \]
\item
The reciprocity sheaf $\Omega ^i _-$ (\cite[Appendix]{KSY}) is unipotent.
When $k$ is perfect, so is $\Omega ^i _{-/k}$.
In fact, every $\mathcal{O}$-module $F$ on $\mathbf{Sm}/k$ 
satisfies the condition
$\H_0(F)=0$ even before Zariski-sheafification.
\end{enumerate}
\end{rmk}

The following Corollaries are direct consequences of
Theorem \ref{htpy}.
Recall that an abelian sheaf $F$ is said to be {\it divisible}
if the multiplication-by-$n$ map $F\xrightarrow{n} F$
is surjective as a map of sheaves for any positive integer $n$.


\begin{cor}\label{cor1}
Let $F$ be a reciprocity sheaf.
\begin{enumerate}[{\rm (1)}]
\item
Suppose $\chark (k)=0$.
Then the unipotent part $\operatorname{U}(F)$ 
is divisible.

\item
Suppose $\chark (k) =p>0$.
Then the unipotent part $\operatorname{U}(F)$ is of $p$-primary torsion.

\end{enumerate}
\end{cor}

\begin{proof}
We first show (1). 
Let $G=\operatorname{U}(F)$.
Consider the cokernel $G/n$ of the map $G\xrightarrow{n} G$.
By Theorem \ref{htpy}, $G/n$ is homotopy invariant.
Thus we have a surjection $\H_0(G)\to \H_0(G/n)=G/n$, and hence $G/n=0$ by $\H_0(G)=0$.


We now show (2). 
Let $\operatorname{U}(F)\{p\}$ be the subsheaf of $\operatorname{U}(F)$ of $p$-primary torsion and let $G$ be the quotient of $\operatorname{U}(F)$ by $\operatorname{U}(F)\{p\}$. Then we have a short exact sequence
\[
0\lto \operatorname{U}(F)\{p\} \lto \operatorname{U}(F)\lto G\lto 0.
\]
By Theorem \ref{htpy}\,(2), $G$ is homotopy invariant. The argument given above applies to show that $G=0$, completing the proof.
\end{proof}



\begin{cor}\label{cor3}
Let $F$ be a reciprocity sheaf.
\begin{enumerate}[{\rm (1)}]
\item
If $\chark (k) =0$, then $\H_1(F)$ is torsion free and $\H_i(F)$ is uniquely divisible for any $i\geq 2$;

\item
If $\chark (k) =p>0$, then $\H_i(F)$ is of $p$-primary torsion for any $i\geq 1.$
\end{enumerate}
\end{cor}

\begin{proof}
Consider the exact sequence
$0\to \operatorname{U}(F)\to F \to \H_0(F) \to 0.
$
Taking homology $\H_*$ gives a long exact sequence
\[
\cdots \to \H_{i+1}(\H_0(F))\to \H_i(\operatorname{U}(F))\to \H_i(F) \to \H_i(\H_0(F)) \to\cdots.
\]
Since $\H_0(F)$ is homotopy invariant, $\H_i(\H_0(F))=0$ for $i\geq1$ and thus we have $\H_i(\operatorname{U}(F))=\H_i(F)$ for $i\geq 1$.  We may therefore assume that $F=\operatorname{U}(F)$. In this case, assertion (2) is clear, since all sections of $\operatorname{U}(F)$ are of $p$-primary torsion by Corollary \ref{cor1}\,(2).
It remains to show the assertion (1) when $F=\operatorname{U}(F)$. Let $n>1$. 
By Corollary \ref{cor1}\,(1), we have an exact sequence
\begin{equation}\label{eqStar} \qquad 0\lto F[n] \lto F\stackrel{ n}{\lto} F \lto 0.
\end{equation}
By Theorem \ref{htpy}, $F[n]$ is homotopy invariant, and hence $\H_i(F[n])=0$ for $i\geq1$. Taking homology on  \eqref{eqStar} gives then a short exact sequence 
$0\lto \H_1(F)\stackrel{n}{\lto} \H_1(F) \to F[n] \lto0,
$
proving that $\H_1(F)$ is torsion free and  that  $\H_i(F)$ is uniquely divisible for $i\geq2$. 
\end{proof}

\begin{rmk}\label{uniquely-div}
We call a sheaf $F$ {\it uniquely divisible} if $F\xrightarrow{n} F$ is an isomorphism of sheaves for every $n>0$ (equivalently, the sections of $F$ are uniquely divisible abelian groups).
For a given reciprocity sheaf $F$ over a field of characteristic $0$,
there are equivalent conditions (see Corollaries \ref{cor1}--\ref{cor3}):
(1) $\operatorname{U}(F)$ is torsion free;
(1)$'$ $\operatorname{U}(F)$ is uniquely divisible;
(2) $F_{\mathrm{tors}}\cong \H_0(F)_{\mathrm{tors}}$ by the canonical map;
(3) $\H_1(F)$ is divisible;
(3)$'$ $\H_1(F)$ is uniquely divisible.

Note that the class of such presheaves with transfers is closed under extension.
\end{rmk}

\begin{rmk}
An example of a unipotent reciprocity sheaf over a field of characteristic zero which is 
{\it not} uniquely divisible is provided by $\mathbb{G}_a/\mathbb{Z}$,
the quotient of $\mathbb{G}_a$ by the constant sub-presheaf with transfers $\mathbb{Z}$.
Its torsion part is the constant sheaf with transfers $\mathbb{Q}/\mathbb{Z}$. In this case one has $\H_1(\mathbb{G}_a/\mathbb{Z})=\mathbb{Z}$.
\end{rmk}
For unique divisibility,  we have the following
\begin{prop}
Suppose $k$ is an algebraically closed field of characteristic zero.
Let $U$ be a unipotent reciprocity sheaf which is an \'etale sheaf.
Then the quotient presheaf (which is a Zariski sheaf)
$U/U(k)$
is uniquely divisible. Here we view $U(k)$ as a constant subsheaf of $U$.

Consequently, over $k$, the sheaf $U$ is uniquely divisible 
if and only if the abelian group $U(k)$ is torsion free.
\end{prop}

\begin{proof}
For all local smooth scheme $X$, 
we have a commutative diagram
of exact sequences
\[\xymatrix{
0\to U(X)_{\mathrm{tors}} \ar[r]  &U(X) \ar[r] &U(X)/U(X)_{\mathrm{tors}}\to 0\\
0\to U(k)_{\mathrm{tors}} \ar[u]^{\rho} \ar[r] &U(k) \ar[u]\ar[r] &U(k)/U(k)_{\mathrm{tors}}\ar[u]\to 0 .
}
\]
By Suslin's rigidity \cite[Theorem 7.20]{MVW}
which is applicable by Theorem \ref{htpy}$(1)$,
the map $\rho$ is an isomorphism.
Now by Corollary \ref{cor1}(1), the groups 
$U(X)/\mathrm{tors}$ and $U(k)/\mathrm{tors}$ are uniquely divisible.
Then by the snake lemma we see that $U(X)/U(k)$ is uniquely divisible.

\end{proof}


\bibliography{bib} 
\bibliographystyle{siam}

\end{document}